\documentclass[11pt]{amsart}
\usepackage{amsmath,amsfonts,latexsym,amssymb,amsthm, verbatim}
\usepackage{url}
 \usepackage[OT2, T1]{fontenc}

\newtheorem{theorem}{Theorem}
\newtheorem{proposition}[theorem]{Proposition}
\newtheorem{lemma}[theorem]{Lemma}
\newtheorem{corollary}[theorem]{Corollary}

\theoremstyle{definition}
\newtheorem{remark}[theorem]{Remark}

\newtheorem*{acknowledgements}{Acknowledgments}

\newcommand{\Q}{\mathbb{Q}}

\newcommand{\Z}{\mathbb{Z}}

\DeclareMathOperator{\Ker}{Ker}

\DeclareMathOperator{\Gal}{Gal}
\DeclareMathOperator{\GL}{GL}
\DeclareMathOperator{\id}{id}

\DeclareMathOperator{\coker}{coker}
\DeclareMathOperator{\im}{Im}

\def\diam#1{\langle#1\rangle}

\begin{document}

\title[The number of twists with large torsion of an ellitpic curve ]{The number of twists with large torsion of an ellitpic curve}
\author{Filip Najman}
\address{Department of Mathematics\\ University of Zagreb\\ Bijeni\v cka cesta 30\\ 10000 Zagreb\\ Croatia}
\email{fnajman@math.hr}
\thanks{The author was supported by the Ministry of Science, Education, and Sports, Republic of Croatia, grant 037-0372781-2821.}
\keywords{Elliptic curves, torsion subgroups, twists}
\subjclass[2010]{11G05}
\begin{abstract}
For an elliptic curve $E/\Q$, we determine the maximum number of twists $E^d/\Q$ it can have such that $E^d(\Q)_{tors}\supsetneq E(\Q)[2]$. We use these results to determine the number of distinct quadratic fields $K$ such that  $E(K)_{tors}\supsetneq E(\Q)_{tors}$. The answer depends on $E(\Q)_{tors}$ and we give the best possible bound for all the possible cases.
\end{abstract}
\maketitle

\section{Introduction}

Let $E/ K$ be an elliptic curve. It is well known, by the Mordell-Weil theorem, that $E(K)$ is a finitely generated abelian group and can hence be written as the direct product of its torsion subgroup and $r$ copies of $\Z$, where $r$ is the rank of $E/K$. We will be interested mostly in the case $K=\Q$, but we will write statements in greater generality whenever possible.

By Mazur's torsion theorem \cite{maz1}, $E(\Q)_{tors}$ is isomorphic to one of the following 15 groups:
\begin{equation}
\vcenter{\openup\jot\halign{$\hfil#\hfil$\cr
C_m, 1 \leq m\leq 12,\ m\neq 11,\cr
C_2  \oplus C_{2m}, \ 1 \leq m\leq 4,\cr}}
\label{rac}
\end{equation}
where $C_m$ is a cyclic group of order $m$. Denote by $E^d$, where $d\in K^*/(K^*)^2$, a quadratic twist of $E/K$. We will assume throughout the paper that $d$ is nonsquare, i.e. that $E^d$ is not isomorphic to $E$ over $K$.

It is natural to ask how the Mordell-Weil group $E(K)$ changes upon (quadratic) twisting. How the rank changes in families of quadratic twists is a subject that has been written extensively about (see \cite{gm,kmr,mr} and the references therein).

It is a trivial fact the $2$-torsion $E(K)[2]$ of $E(K)$ does not change upon quadratic twisting. We will say that a quadratic twist $E^d/K$ has \textit{large} torsion if $E^d(K)_{tors}\neq E(K)[2]$.

Another fact is that in a family of quadratic twists, $E^d(K)_{tors}$ will consist only of the even order torsion points for all but finitely many $d\in K^*/(K^*)^2$ (see \cite[Lemma 5.5]{mr}, \cite[Proposition 1]{gm}). This statement follows easily from the Uniform Boundedness Conjecture, which has been proven by Merel \cite{mer}.

The main purpose of this paper is to make this statement effective for $K=\Q$, by determining the maximum number of quadratic twists with large torsion an elliptic curve can have over $\Q$. We will show that this bound is $3$ for a general $E/\Q$, and it can be even smaller, depending on $E(\Q)_{tors}$.

\begin{theorem}
The possible number of quadratic twists $E^d(\Q)$ with large torsion of $E/ \Q$, depending on $E(\Q)_{tors}$ is as in the table below.
\label{main}
\begin{center}
\begin{tabular}{|c|c||c|c||c|c||c|c||c|c|}
\hline
$C_1$ & $0,1,2$ & $C_4$ & $1,2$ &$C_{7}$& $0$ & $C_{10}$& $0$ & $C_{2}\oplus C_4$& $0,1$ \\
\hline
$C_2$ & $0,1,2,3$ & $C_5$ & $0,1$ &$C_{8}$& $1$ & $C_{12}$& $1$ & $C_{2}\oplus C_6$& $0$ \\
\hline
$C_3$ & $0,1$ & $C_6$ & $0,1,2$ &$C_{9}$& $0$ & $C_{2}\oplus C_2$& $0,1$ & $C_{2}\oplus C_8$& $0$ \\
\hline

\end{tabular}
\end{center}
\end{theorem}

We prove this result in Sections $3$ and $4$. Explicit examples for all the cases, apart from those in which $E(\Q)_{tors}$ uniquely determines the number of twists with large torsion, can be found in Table 1 at the end of the paper.

In Section $5$ we look at a related problem: given an elliptic curve $E/ \Q$, how many quadratic extensions $K/ \Q$ do there exist such that $E(K)_{tors}\supsetneq E(\Q)_{tors}$? Gonz\'alez-Jim\'enez and Tornero \cite{gt} gave upper bounds for the possible number of such quadratic extensions, depending on $E(\Q)_{tors}$, but their bounds are far from optimal. We give the best possible bounds in the following theorem.

\begin{theorem}
The possible numbers of quadratic fields $K/ \Q$ such that $E(K)_{tors}\supsetneq E(\Q)_{tors}$, depending on $E(\Q)_{tors}$, is given in the table below.
\label{main2}
\begin{center}
\begin{tabular}{|c|c||c|c||c|c||c|c||c|c|}
\hline
$C_1$ & $0,1,2$ & $C_4$ & $1,2,3$ &$C_{7}$& $0$ & $C_{10}$& $1$ & $C_{2}\oplus C_4$& $0,1,2,3$ \\
\hline
$C_2$ & $1,2,3,4$ & $C_5$ & $0,1$ &$C_{8}$& $1,3$ & $C_{12}$& $1$ & $C_{2}\oplus C_6$& $0,1$ \\
\hline
$C_3$ & $0,1$ & $C_6$ & $1,2,3$ &$C_{9}$& $0$ & $C_{2}\oplus C_2$& $0,1,2,3$ & $C_{2}\oplus C_8$& $0$ \\
\hline

\end{tabular}
\end{center}
\end{theorem}

Explicit examples for all the cases, apart from those in which $E(\Q)_{tors}$ uniquely determines the number of quadratic fields in which the torsion grows, can be found in Table 2 at the end of the paper.

\section{Auxiliary results}

In this section we fix notation and list the main tools, most of which are well known, which we will later use to prove Theorems \ref{main} and \ref{main2}.

Let $E[n]=\{ P\in E(\overline \Q )| nP=0 \}$ denote the $n$-th division group of $E$ over $\overline \Q$ and let $\Q(E[n])$ be the $n$-th division field of $E$. The Galois group $\Gal (\overline \Q /\Q)$ acts on $E[n]$ and gives rise to an embedding $$\rho_n: \Gal(\Q (E[n])/\Q)  \hookrightarrow \GL_2 (\Z/n\Z)$$
called the \emph{mod $n$ Galois representation}. For a number field $K$, $E(K)[n]$ denotes the set of $K$-rational points in $E[n]$.

If there exists a $K$-rational cyclic isogeny $\phi:E\rightarrow E'$ of degree $n$, this implies that $\Ker \phi$ is a $\Gal(\overline K / K)$-invariant cyclic group of order $n$ and we will say that $E/K$ has an $n$-isogeny.

When we say that an elliptic curve over a field $K$ has $n$-torsion we mean that it contains a subgroup isomorphic to $C_n$.

We start with the following important lemma.

\begin{lemma}[Lemma 1.1 in \cite{ll}]
\label{lem1}
Let $L/K$ be a quadratic extension of number fields and let $L=K(\sqrt d)$.
There exist homomorphisms
$$f:E(K)\oplus E^d(K) \rightarrow E(L),$$
$$g: E(L)\rightarrow E(K)\oplus E^d(K),$$
such that the kernels and cokernels of $f$ and $g$ are annihilated by $[2]$.
\end{lemma}

For odd order torsion Lemma \ref{lem1} translates to the following.

\begin{corollary} Let $n$ be an odd integer. Using the same notation as in Lemma \ref{lem1},
$$E(K)[n]\oplus E^d(K)[n] \simeq E(L)[n].$$
\label{zbrajanje}
\end{corollary}

Odd order torsion over quadratic extensions will induce an isogeny of the same order over the base field.

\begin{lemma}[Lemma 5 in \cite{naj}]
\label{spust}
Let $L/ K$ be a quadratic extension, $n$ an odd positive integer, and $E/ K$ an elliptic curve such that $E(L)$ contains $\Z /n\Z$. Then $E/ K$ has an $n$-isogeny.
\end{lemma}

Note that in the statement of \cite[Lemma 5]{naj}, the base field is $K=\Q$, but that everything generalizes trivially to a general number field $K$.

It will be important for us to know which are the possible $n$-isogenies (by which we mean a cyclic isogeny of degree $n$) over $\Q$ and which are the possible torsion groups over quadratic fields.

\begin{theorem}[\cite{ken2,ken3,ken4,maz2}]
\label{izogklas}
Let $E/ \Q$  be an elliptic curve with an $n$-isogeny. Then $n\leq 19$ or $n\in \{21,25,27,37,43,67,163\}$.
\end{theorem}

\begin{theorem}[\cite{Kam1,km}]
Let $E/K$ be an elliptic curve over a quadratic field $K$. Then $E(K)_{tors}$ is isomorphic to one of the following groups.
\begin{equation}
\vcenter{\openup\jot\halign{$\hfil#\hfil$\cr
C_m, 1 \leq m\leq 18,\  m\neq 17,\cr
C_2 \oplus C_{2m},\ 1 \leq m\leq 6,\cr
C_3 \oplus C_{3m} ,\ m=1,2,\cr
C_4 \oplus C_4.\cr}}
\end{equation}
\label{kvad}
\end{theorem}

We denote by $\psi_n$ the $n$-th division polynomial of an elliptic curve $E$ (see \cite[Section 3.2]{was} for details), which satisfies that, for a point $P\in E$, $\psi_n(x(P))=0$ if and only if $nP=0$. A method we will often use throughout the paper to determine whether a quadratic twist of a given elliptic curve has large torsion is to factor $\psi_n$ and then compute what the torsion is in the extensions generated by factors of degree $2$ and in the fields of definition of the $y$-coordinate of the point $(x,y)$, where $x$ is the root of a linear factor of $\psi_n$. 

\section{Quadratic twists of curves with odd order torsion}

The easier case will be when an elliptic curve has torsion of odd order. We will use the names of particular elliptic curves as they appear in \cite{cre}.

\begin{proposition}
\label{neparno}
If $E(\Q)_{tors}\simeq C_n$, where $n$ is odd, then $E^d(\Q)_{tors}$ has odd order for all $d\in \Q^*/(\Q^*)^2$. Furthermore,

\begin{itemize}
\item[a)] If $E(\Q)_{tors}\simeq C_n$, where $n>5$ is odd, then $E^d(\Q)_{tors}\simeq C_1$ for all $d\in \Q^*/(\Q^*)^2$.

\item[b)] If $E(\Q)_{tors}\simeq C_5$, then $E^d(\Q)_{tors}\simeq C_1$ for all $d\in \Q^*/(\Q^*)^2$, except if
    $E$ is the curve $50b1$ and $d_1=5$ or if $E$ is $50b2$ and $d_1=-15$; in both of these cases $E^{d_1}(\Q)_{tors}\simeq C_3$ and $E^d(\Q)_{tors}\simeq C_1$ for all other $d\in \Q^*/(\Q^*)^2$.

\item[c)] If $E(\Q)_{tors}\simeq C_3$, then there are $3$ cases:
    \begin{itemize}
    \item[i)] $E^d(\Q)_{tors}\simeq C_1$ for all $d\in \Q^*/(\Q^*)^2$.
    \item[ii)] $E^{-3}(\Q)_{tors}\simeq C_3$ and $E^d(\Q)_{tors}\simeq C_1$ for all $-3\neq d\in \Q^*/(\Q^*)^2$.
    \item[iii)] $E$ is the curve $50a3$ and $d_1=5$ or $E$ is $450b4$ and $d_1=-15$; in both of these cases $E^{d_1}(\Q)_{tors}\simeq C_5$ and
        $E^d(\Q)_{tors}\simeq C_1$ for all other $d\in \Q^*/(\Q^*)^2$.
    \end{itemize}

\item[d)] If $E(\Q)_{tors}\simeq C_1$, then there are the following cases:
    \begin{itemize}
    \item[i)] $E^d(\Q)_{tors}\simeq C_1$ for all $d\in \Q^*/(\Q^*)^2$.

    \item[ii)] $E^{d_1}(\Q)_{tors}\simeq C_n$, where $n=7$ or $9$ for some  $d_1\in \Q^*/(\Q^*)^2$. Then $E^d(\Q)_{tors}\simeq C_1$ for all other $d_1\neq d\in \Q^*/(\Q^*)^2$.

    \item[ii)] There exist $2$ quadratic twists $E^{d_1}$ and $E^{d_2}$ such that $E^{d_1}(\Q)_{tors}\simeq E^{d_2}(\Q)_{tors}\simeq C_3$. Then for all other $d\in \Q^*/(\Q^*)^2$, it holds that  $E^d(\Q)_{tors}\simeq C_1$.

    \item[iii)] There exists $1$ quadratic twist such that  $E^{d_1}(\Q)_{tors}\simeq C_3$ and for all other $d\in \Q^*/(\Q^*)^2$, it holds that  $E^d(\Q)_{tors}\simeq C_1$.

    \item[iv)] If $E$ is a twist of $50b1$ or $50b3$, and is not $50b2$ or $450b4$, respectively, then $E$ has one twist such that $E^{d_1}(\Q)_{tors}\simeq C_3$ and one twist such that $E^{d_2}(\Q)_{tors}\simeq C_5$ and for all other $d\in \Q^*/(\Q^*)^2$, it holds that  $E^d(\Q)_{tors}\simeq C_1$. These are the only instances where an elliptic curve can have $2$ quadratic twists of different odd order torsion.

    \item [v)]  If $E$ is not a twist of $50b1$ or $50b3$ and $E^{d_1}(\Q)_{tors}\simeq C_5$, for some $d_1 \in \Q^*/(\Q^*)^2$, then for all other $d\in \Q^*/(\Q^*)^2$, it holds that $E^d(\Q)_{tors}\simeq C_1$.
    \end{itemize}
\end{itemize}
\end{proposition}
\begin{proof} \textbf{a)} If there existed a quadratic twist $E^d$ such that $E^d(\Q)_{tors} \simeq C_m$ for some odd $m$, then it would hold, by Corollary \ref{zbrajanje} that
$$E(\Q(\sqrt d))_{tors}\simeq E(\Q)_{tors}\oplus E^d(\Q)_{tors},$$
which is a contradiction with Theorem \ref{kvad}.

\textbf{b)} If $E$ had a quadratic twist with torsion $C_n$, where $n\geq5$ is odd, Corollary \ref{zbrajanje} would give a contradiction with Theorem \ref{kvad}. The only cases when an elliptic curve $E^d$ with torsion $C_5$ has a quadratic twist with torsion $C_3$ is when $E(\Q(\sqrt d))_{tors}\simeq C_{15}$, and there are only $2$ such cases (see \cite[Theorem 2 c)]{naj}).

\textbf{c)} As before, $E$ cannot have a quadratic twist with torsion $C_n$ where $n>5$ because of Corollary \ref{zbrajanje} and Theorem \ref{kvad}. As in b) there are only $2$ cases when a twist can have torsion $C_5$ \cite[Theorem 2 c)]{naj}. There exist elliptic curves $E/ \Q$ with torsion $E(\Q)_{tors}\simeq C_3$ with a twist $E^d(\Q)_{tors}\simeq C_3$. For example, all curves of the form
$$x^3+y^3+z^3=3\alpha xyz,\ \alpha^3 \neq 1$$
have this property (see \cite[Table 1.]{kub}). Likewise, there exist elliptic curves with $3$-torsion and all quadratic twists with trivial torsion, for example $19a3$ \cite[Table 2]{gt}. The case iii) follows from Corollary \ref{zbrajanje} and \cite[Theorem 2 c)]{naj}.

\textbf{d)} The case i) is for example 11a2 \cite[Table 2]{gt}, while the cases ii), iii), iv), v) and the impossibility of any other case follow from a), b) and c).
\end{proof}

\section{Quadratic twists of curves with even order torsion}

\begin{lemma}
\label{4torz}
If $E(\Q)_{tors}\simeq C_{4n}$, then there exists exactly one quadratic twist $E^d$ such that $E^d(\Q)\supset C_4$.
\end{lemma}
\begin{proof}
There exists exactly one quadratic field $\Q(\sqrt d)$ over which $E$ has full 2-torsion; the elliptic curves $E$ and $E^d$ are isomorphic over this field and $E(\Q(\sqrt d))\simeq E^d(\Q(\sqrt d))$ $\supset C_2 \oplus C_4$. By \cite[Theorem 5 iii)]{gt} it follows that $E^d(\Q) \supset C_4$. Hence we have proven that there exists one quadratic twist with torsion $C_4$; it remains to prove it is the only one. Suppose there exists a different one $E^{d_1}$. Since $E(\Q(\sqrt {d_1}))\simeq C_{4m}$, for some integer $m$, from Lemma \ref{lem1} we have the exact sequence
$$0 \rightarrow \Ker \phi \rightarrow E(\Q(\sqrt d))\xrightarrow{\phi} E(\Q)\oplus E^d(\Q) \rightarrow \coker \phi \rightarrow 0.$$
As $\coker \phi$ is a quotient of a group isomorphic to the quotient of $C_{4n}\oplus C_{4m}$ by $\im \phi$, which is a subgroup of $C_{4k}$ for some integer $k$. Since $C_{4n}$ and $C_{4m}$ have to be subgroups of $C_{4k}$, the only possibility that the cokernel is annihilated by $2$ (which is necessary by Lemma \ref{lem1}) is that $k=n=m=1$ and that the quotient is isomorphic to $C_2 \oplus C_2$.

But this also leads to a contradiction, as the quotient of $C_{4}\oplus C_4$ by $C_4$ cannot be isomorphic to $C_2 \oplus C_2$, and hence annihilated by 2, which contradicts Lemma \ref{lem1}.
\end{proof}

We use this lemma to count the possible number of quadratic twists with large torsion of an elliptic curve with even order torsion.

\begin{proposition}
\label{parno} Let $E(\Q)_{tors}\simeq C_{2n}$.
\begin{itemize}

\item[a)] If $n=1$ then the following $5$ cases are possible:
    \begin{itemize}
    \item[i)] All the quadratic twists of $E$ have just $2$-torsion.

    \item[ii)] $E$ has $1$ quadratic twist with a point of odd order and every other quadratic twist has only $2$-torsion.

    \item[iii)] $E$ has $2$ quadratic twists with $3$-torsion and every other quadratic twist has only $2$-torsion.

    \item[iv)]  $E$ has $2$ quadratic twists with $4$-torsion and every other quadratic twist has only $2$-torsion.

    \item[v)] $E$ has $2$ quadratic twists with $4$-torsion and $1$ quadratic twist with $3$-torsion, and every other quadratic twist has only $2$-torsion.

    \end{itemize}

\item[b)] If $n=2$ then $E$ has $1$ quadratic twist with $4$-torsion. The curve $E$ has $0$ or $1$ quadratic twists with $3$-torsion. All other quadratic twists have only $2$-torsion.

\item[c)] If $n=3$ then there are $3$ cases:
\begin{itemize}
    \item[i)] All the quadratic twists have just $2$-torsion.

    \item[ii)] $E$ has $1$ quadratic twist with $6$-torsion and every other quadratic twist has $2$-torsion.

    \item[iii)] $E$ has $2$ quadratic twists with $4$-torsion and every other quadratic twist has only $2$-torsion.

    \end{itemize}

\item[d)] If $n=4$ then $E$ has exactly one quadratic twist which has torsion $C_4$ and all the other twists have just $2$-torsion.

\item[e)] If $n=5$ then all the quadratic twists of $E$ have just $2$-torsion.

\item[f)] If $n=6$ then $E$ has exactly one quadratic twist which has $4$-torsion and all the other quadratic twists have just $2$-torsion.

\end{itemize}
\end{proposition}
\begin{proof}
\textbf{a)} To prove that each of these cases is possible, we simply list the curves which fall into the respective case. Examples for each of the cases can be found in Table 1. That they have the appropriate number of quadratic twists with given torsion is easy to determine by checking their division polynomials or by checking their isogeny diagrams in \cite{cre}. 

As can be seen from Proposition \ref{4torz}, $E$ can have either $0$ or $2$ quadratic twists with $4$-torsion. From Corollary \ref{zbrajanje} and Theorem \ref{kvad}, it can be seen that if $E$ has a point of order $5$, then it cannot have any other quadratic twist with large torsion. Using again Corollary \ref{zbrajanje} and Theorem \ref{kvad}, we see that $E$ can have at most $2$ quadratic twists with a point of order $3$.  Note that if it has $2$ quadratic twists with a point of order 3, then it cannot have a quadratic twist with order $4$, because in this case $E$ would have $2$ independent isogenies of degrees $3$ and $12$. This would imply that $E$ is isogenous to a curve with a $36$-isogeny, which is a contradiction with Theorem \ref{izogklas}.

The cases \textbf{b)}, \textbf{c)} and \textbf{e)} are proven in the same manner as \textbf{a)}, using Corollary \ref{zbrajanje} and Theorems \ref{izogklas} and \ref{kvad}.

In the case \textbf{d)} one has to additionally check that the quadratic twist of $E$ which contains $4$-torsion has torsion isomorphic to $C_4$ (and not $C_8$). But that can be seen from the fact (Lemma \ref{lem1}) that the kernel of the map $E^d(\Q)\oplus E(\Q)\rightarrow E(\Q(\sqrt d))$ has to be annihilated by 2, and that $(C_8\oplus C_8)/(C_2\oplus C_2) \simeq C_4 \oplus C_4$, which leads to a contradiction with $ E(\Q)\subset  E(\Q(\sqrt d))$.

\end{proof}

Finally, we list the possibilities when $E(\Q)$ has full 2-torsion, which were proven by Kwon.

\begin{proposition}[\cite{kwo}, Theorem 2]
\label{full2tors}
Let $E(\Q)_{tors}\simeq C_2\oplus C_{2n}$.

\begin{itemize}
\item[i)] If $n=1$, then $E$ can have either no large torsion quadratic twists, or $1$ or $2$ quadratic twists with torsion $C_2\oplus C_{4}$ and no other large torsion quadratic twists, or $1$ quadratic twist with torsion $C_2\oplus C_{2m}$ for $m=3$ or $4$ and no other large torsion quadratic twists.

\item[ii)] If $n=2$, then $E$ can have $0$ or $1$ quadratic twists with torsion $C_2\oplus C_{4}$.

\item[iii)] If $n=3$ or $4$ then $E$ has no quadratic twists with large torsion.

\end{itemize}
\end{proposition}

Propositions \ref{neparno}, \ref{parno} and \ref{full2tors} together prove Theorem \ref{main}.

\medskip

Finally, we will prove that there are finitely many large torsion quadratic twists over any number field. This is a generalization of \cite[Proposition 1]{gm}, where it was shown that there are finitely many large torsion quadratic twists over $\Q$ and \cite[Lemma 5.5]{mr}, where it was shown that there are finitely many quadratic twists with a point of odd order over a general number field $K$. We will simultaneously prove that $E(K)_{tors}$, where $K$ is a number field, grows in finitely many quadratic extensions, generalizing \cite[Theorem 6]{gt} and \cite[Lemma 3.4 a)]{jkl}.
\begin{theorem}
\label{finiteness}
Let $E/K $ be an elliptic curve over a number field $K$.
\begin{itemize}
\item[a)] There exists finitely many quadratic extensions $L/K$ such that\\ $E(L)_{tors}\supsetneq E(K)_{tors}.$

\item[b)]  There exists finitely many $E^d/K$ with large torsion.
\end{itemize}
\end{theorem}
\begin{proof} \textbf{a)} By \cite{mer}, there exists a bound $n$ such that $\# E(L)_{tors}<n$ for all quadratic extensions $L$ of $K$. It suffices to prove that for any $n_1<n$ there will be finitely many quadratic extensions $L$ of $K$ such that there exists a point of order $n_1$ over $L$, but which is not defined over $K$. But this follows trivially from the fact that $E[n_1]$ has $n_1^2$ (and hence finitely many) elements.

\textbf{b)} By \cite[Lemma 5.5]{mr}, it suffices to prove that there exist only finitely many quadratic twist $E^d(K)$ does not contain $C_4$. First note that if $E(K)_{tors}=E(K(\sqrt d))_{tors}$, then $E^d(K)$ cannot contain $C_4$, for otherwise the map $E^d(K)\oplus E(K) \rightarrow E(K(\sqrt d))$ from Lemma \ref{lem1} could not be annihilated by $2$. Now from a), it follows that this assumption is true for all but finitely many $d\in K^*/(K^*)^2$, proving the claim.
\end{proof}

\section{Number of quadratic fields where the torsion grows}

We introduce the following notation: for $E/ \Q$ by $g(E)$ we denote the number of quadratic fields $K$ such that $E(K)_{tors}\supsetneq E(\Q)_{tors}$.

We first deal with the case when the order of the torsion is odd.

\begin{proposition}
Let $E/ \Q$ such that $|E(\Q)_{tors}|$ is odd. Then $g(E)$ is equal to the number of quadratic twists of $E$ with large torsion.
\end{proposition}
\begin{proof}
This follows directly from Corollary \ref{zbrajanje}.
\end{proof}

Determining $g(E)$ when $|E(\Q)_{tors}|$ is even is much harder. We start with the following lemma.

\begin{lemma}
If $E(\Q)_{tors}\simeq C_{2n}$, $n\neq 2$, then there are either $0$ or $2$ quadratic fields $K$ such that $E(K)_{tors}\simeq C_{4n}$. If $n=2$, then there can be either $0$, $1$ or $2$ such quadratic fields, and if there is exactly one such quadratic field then $g(E)=1$.
\label{lem2}
\end{lemma}
\begin{proof} We first look at the case $n\neq 2$.
Let $P$ be the generator of $E(\Q)_{tors}$. Now the equation $2Q=P$ has $4$ solutions in $E(\overline \Q)$ and $\Gal(\overline \Q / \Q)$ acts on the solutions. Since the solutions are by assumption not defined over $\Q$, it follows that they have either one orbit of length $4$ or $2$ orbits of length $2$ under the action of $\Gal(\overline \Q / \Q)$. In the former case, there will not exist a quadratic field $K$ such that $E(K)\supset C_{4n}$. In the latter case it follows that $\Gal(\overline \Q / \Q)$ acts on the solutions through $2$ automorphisms of $\overline \Q$ of order 2, say $\sigma$ and $\tau$. It remains to show that $\sigma\neq \tau$ or equivalently that all the solutions of $2Q=P$ are defined over the same field $K$, giving $E(K)\supset C_2\oplus C_4$.

Suppose $\sigma= \tau.$ It follows that all the solutions of $2Q=P$ are defined over a quadratic field $K$, from which it follows that $E(K)\supset C_2 \oplus C_{4n}$. The cases $n=1$ and $n=3$ are impossible by \cite[Theorem 5 iii)]{gt}, while the cases $n\geq 4$ are impossible by Theorem \ref{kvad}.

In the case $n=2$, it is possible that $\sigma= \tau$; the curve 1344m5 is such a curve. It remains to prove that in the case $\sigma= \tau$, or equivalently $E(K)_{tors}\simeq C_2\oplus C_8$ for some quadratic field $K$, that for all quadratic fields $F\neq K$, it is true that $E(F)_{tors}\simeq C_4$. It is obvious that $E$ cannot gain even order torsion in any quadratic field apart from $K$, so it remains to show that it does not gain any odd order torsion in any quadratic extension. First note that such a curve has a Galois-invariant subgroup of order 8 (say the one generated by a $Q$ satisfying $2Q=P$). If it gained $n$-order torsion, where $n$ is odd, in a quadratic extension, then it would mean that over $\Q$ it has an $n$-isogeny. Now it follows that there exists an elliptic curve with an $8n$-isogeny over $\Q$, which is in contradiction with Theorem \ref{izogklas}.
\end{proof}

\begin{proposition}
\label{2torsprosirenje}
Let $E(\Q)_{tors}\simeq C_{2n}$.

\begin{itemize}
\item[a)] If $n=1$, then $g(E)=1,2,3$ or $4$.

\item[b)] If $n=2$ then $g(E)=1,2$ or $3$.

\item[c)] If $n=3$ then $g(E)=1,2$ or $3$.

\item[d)] If $n=4$ then $g(E)=1$ or $3$.

\item[e)] If $n=5$ or $6$ then $g(E)=1$.
\end{itemize}
\end{proposition}
\begin{proof}
\textbf{a)} Lemma \ref{lem2} shows that $E$ can have $0$ or $2$ extensions over which $E$ contains $C_4$, and there has to be exactly $1$ extension over which $E$ has torsion $C_2\oplus C_2$ and this extension is necessarily different from the ones over which $E$ has torsion $C_4$ (\cite[Theorem 5 iii)]{gt}). In addition, there can exist a quadratic field over which $E$ has torsion $C_6$. Examples of curves $E$ such that $g(E)=1,2,3$ and $4$ can be found in Table 2.


It is impossible that $E$ gains $4$-torsion in a quadratic extension and odd order torsion in 2 quadratic extensions (of degree $m$ and $n$, respectively). To see that, first note that, for a quadratic field $K$, if $E(K)$ contained $C_4$, (but not $C_2 \oplus C_4$, since that is impossible), then $\Gal(K/\Q)$ would act on $E[4](K)$ by permuting the $2$ points of order $4$, and hence $E$ would have a 4-isogeny over $\Q$. The curve $E$ would also have an $m$-isogeny and an (independent) $n$-isogeny, and hence there would exist an elliptic curve with a $4mn$-isogeny (see \cite[Lemma 7]{naj}) over $\Q$, which is impossible by Theorem \ref{izogklas}. This proves that $g(E) \leq 4$.


\textbf{b)} In the case $n=2$, it follows by Lemma \ref{lem2} that the torsion of $E$ contains $C_8$ in $0, 1$ or $2$ quadratic fields. If it contains $C_8$ in one extension, then as we have shown $g(E)=1$.

If it contains $C_8$ in $2$ extensions, then $E$ has an $8$-isogeny over $\Q$ (this can be shown using the same argument as in a)), and from Theorem \ref{izogklas} it follows that there are no extensions where $E$ gains points of odd order because otherwise Lemma \ref{spust} would lead to a contradiction with Theorem \ref{izogklas}. In addition there will be one quadratic field, different from the ones where the torsion is $C_8$, where the torsion will be $C_2 \oplus C_4$. In this case $g(E)=3$; for example the curve 15a7 is such a curve.

If $E$ contains $C_8$ in $0$ quadratic fields, then it is possible that there exist $1$ quadratic field where $E$ has torsion $C_{12}$. Together with the field with torsion $C_2 \oplus C_4$, this means that $g(E)=2$; 150c1 is such a curve.

\textbf{c)} By Lemma \ref{lem2}, there are 0 or 2 quadratic fields $K$ such that $E(K)_{tors}\simeq C_{12}$ and there is exactly one quadratic field $K$ such that $E(K)_{tors}\simeq C_{2}\oplus {C_6}$. In addition, it is possible that $E(\Q(\sqrt{-3})\simeq C_3\oplus C_6$. Examples of curves with $g(E)=1,2$ and $3$ can be found in Table 2.

To complete the proof, it remains to show that $g(E)\leq 3$, or in other words, that it is impossible that $E(\Q(\sqrt{-3}))\simeq C_3\oplus C_6$ and that there exists a quadratic field $K$ such that $E(K)_{tors}\simeq C_{12}$. If such a $K$ existed, it would follow that $E/\Q$ has a $12$-isogeny and from the fact that $E(\Q(\sqrt{-3}))\simeq C_3\oplus C_6$ it would follow that $E/\Q$ has 2 independent $3$-isogenies. Combining these 2 facts we would get that $E/\Q$ has a 12-isogeny and an independent 3-isogeny. By \cite[Lemma 7]{naj}, it follows that there exists an elliptic curve over $\Q$, isogenous to $E$, with a 36-isogeny. But this is impossible by Theorem \ref{izogklas}.

\textbf{d)} This follows from the fact that there exists exactly $1$ quadratic field over which $E$ has torsion $C_2\oplus C_{2n}$ and, by Lemma \ref{lem2}, 0 or 2 quadratic extensions with $16$-torsion.

\textbf{e)} As before, by Theorem \ref{kvad}, there exists exactly $1$ quadratic field over which $E$ has torsion $C_2\oplus C_{2n}$.

\end{proof}

\begin{remark}
Note that if $E(\Q)_{tors}\simeq C_2$, it is impossible that the torsion of $E$ grows to $C_4$ in one extension and to $C_{2n}$, where $n\geq 5$ is odd in another extension. One can see this as if $E(\Q)_{tors}\simeq C_2$ and $E(K)_{tors}\supset C_4$ for some quadratic field $K$, it follows that $E$ has a $4$-isogeny over $\Q$. Also, by Lemma \ref{spust}, it follows that $E$ has a $4n$-isogeny over $\Q$, which is a contradiction with Theorem \ref{izogklas}.
\end{remark}

\begin{proposition}
\label{full2torskvad}
Let $E(\Q)_{tors}\simeq C_2\oplus C_{2n}$.

\begin{itemize}
\item[a)] If $n=1$, then $g(E)=0,1,2$ or $3$.

\item[b)] If $n=2$, then $g(E)=0,1,2$ or $3$.

\item[c)] If $n=3$ then $g(E)=0$ or $1$.

\item[d)] If $n=4$ then $g(E)=0$.

\end{itemize}
\end{proposition}

\begin{proof}
\textbf{a)} We first note that $E$ can gain torsion $C_2\oplus C_4$ over either $0,1,2$ or $3$ quadratic fields; there are $12$ points of order $4$ in $E(\overline \Q)$, so it is obvious that there cannot be more than $3$ extensions with torsion $C_2\oplus C_4$, since $E$ gains $4$ points of order $4$ in each of them.

Next we observe that it is impossible that $E$ can gain points of odd order in more than $1$ extension, since Lemma \ref{spust} would give a contradiction with Theorem \ref{izogklas}.

We claim that if $E$ gains torsion $C_2\oplus C_4$ in $3$ extensions, that then $E$ gains no odd order points in any quadratic extension. Let $G=\Gal(\overline \Q /  \Q)$. Suppose the opposite. Then by Lemma \ref{spust} it follows that $E$ has an $n$-isogeny, for some odd value $n$. In this case $K:=\Q(E[4])$ is a biquadratic field and hence the image of the mod $4$ Galois representation $\rho_{E,4}(G)$ has to be isomorphic to $\Gal(K/\Q)\simeq C_2\oplus C_2$. Let $\Gal(K/\Q)=\{\id, \sigma_1, \sigma_2, \sigma_3\}$. All the $\sigma_i$ have to have the property that they fix the complete 2 torsion, as $E(\Q)\supset C_2\oplus C_2$, and that they fix a different point of order $4$, as $E(F_i)\supset C_2\oplus C_4$, where $F_i$ is the quadratic field fixed by $\sigma_i$. By running through (in Magma \cite{mag}) all the subgroups of $\GL_2(\Z/4\Z)$ isomorphic to $C_2\oplus C_2$ we obtain that the only possibility, up to conjugacy, for $\rho_{E,4}(G)$ is the group
$$\left\{ \begin{pmatrix}1 & 0\\ 0 &1\end{pmatrix},\begin{pmatrix}1 & 0\\ 0 &3\end{pmatrix},\begin{pmatrix}1 & 2\\ 0 &1\end{pmatrix}, \begin{pmatrix}1 & 2\\ 0 &3\end{pmatrix}\right\}.$$
We see that this group is contained in a Borel subgroup, hence it follows that $E$ has a cyclic $4$-isogeny over $\Q$. As $E$ also has an independent $2$-isogeny, this implies that there is an elliptic curve $E'/\Q$, $2$-isogenous to $E$ over $\Q$, with an $8n$-isogeny, which is impossible.

Thus, we have proven that $g(E)\leq 3$. Furthermore, one can find examples of curves $E/\Q$ such that $g(E)=0,1,2$ or $3$ in Table 2, proving a).

\textbf{b)} First note that by \cite[Theorem 2]{kwo}, the torsion of $E$ over a quadratic field can grow into either $C_4\oplus C_4$ or into $C_2\oplus C_8$.

We claim that there exist at most two quadratic fields over which $E$ has torsion $C_2 \oplus C_8$. Let $\diam{P,Q}\simeq E(\Q)_{tors}$, where $P$ is of order $2$ and $Q$ of order $4$. If $E$ gains points of order 8 over a quadratic field, it gains 8 of them in each quadratic field. The only way in which this can happen is that a point of order 4 becomes divisible by 2. Over one quadratic field $K$, $E$ can possibly gain the $4$ solutions of the equation $2X=Q$ and the $4$ solutions of the equation $2X=3Q$, and over another field $F$, $E$ can gain the 4 solutions of $2X=P+Q$ and the 4 solutions of $2X=P+3Q$. This exhausts all the possible quadratic points of order 8, proving our claim.


It follows that $g(E)\leq 3$ (at most two extension with torsion $C_2\oplus C_8$ and one with $C_4\oplus C_4$). Examples of curve $E/\Q$ such that $g(E)=0,1,2$ and $3$ are given in Table 2, completing the proof of $b)$.

Part \textbf{c)} is \cite[Theorem 1 (ii)]{kwo} and \textbf{d)} is \cite[Theorem 1 (iii)]{kwo}.
\end{proof}

\section{Torsion of Cubic and Quartic twists}

To complete the study of large torsion in twists, we have to study cubic and quartic twists of elliptic curves. This is a much easier problem, since we have to study only the elliptic curves with $j$-invariant $0$ and $1728$. There is a difference in cubic and quartic twisting compared to quadratic twisting in the sense that the $2$-torsion can change upon quartic and cubic twisting. In particular, the elliptic curve $E_2:y^2=x^3+2$ has trivial torsion, while its cubic twist $E_1:y^2=x^3+1$ and all quadratic twists of $E_1$ have non-trivial $2$-torsion.

We state the results for elliptic curves with $j$-invariant $0$ in the following proposition.

\begin{proposition}
Among all the curves with $j$-invariant $0$, there exists infinitely many with trivial torsion, infinitely many with torsion $C_2$, infinitely many with torsion $C_3$, and 1 with torsion $C_6$.
\end{proposition}
\begin{proof}
An elliptic curve with $j$-invariant $0$ is of the form
$$E_D:y^2=x^3+D, \text{ where } D\in \Q^*/(\Q^*)^6.$$
It follows that $E_D(\Q)[2]\simeq \Z/2\Z$ if $D$ is a cube and
and $E_D(\Q)[2]$ is trivial otherwise. A computation using division polynomials proves that there is no $4$-torsion in $E_D(\Q)$.

By \cite[Theorem 3]{diu} it follows that $E_D(\Q)[3]\simeq \Z/3\Z$ if $D$ is a square and $E_D(\Q)[3]$ is trivial otherwise.

There can be no other torsion in $E_D(\Q)$ by \cite[Proposition 1]{par}.
\end{proof}

\begin{remark}
Note that elliptic curves with $j$-invariant $0$ are the only curves that have infinitely many large torsion twists, if one does not restrict twisting only to quadratic twisting. \end{remark}

We deal in a similar way with elliptic curves with $j$-invariant $1728$.

\begin{proposition}
Among all the curves with $j$-invariant $1728$, there exists infinitely many with torsion $C_2$, infinitely many with torsion $C_2\oplus C_2$ and $1$ with torsion $C_4$.
\end{proposition}
\begin{proof}
An elliptic curve with $j$-invariant $1728$ is of the form
$$E_D:y^2=x^3+Dx, \text{ where } D\in \Q^*/(\Q^*)^4.$$
It follows that $E_D(\Q)[2]\simeq \Z/ 2 \Z$ if $D$ is not a square and $E_D(\Q)[2]\simeq \Z/ 2 \Z\oplus \Z /2\Z$ if $D$ is a square. A computation using division polynomials shows that the only elliptic curve with $4$-torsion in this family is
$E_{4}:y^2=x^3+4x \text{ with }E_4(\Q)_{tors}\simeq \Z/4\Z.$
By \cite[Theorem 3]{diu}, there is no $p$-torsion in $E_D(\Q)$ for any odd primes $p$.

\end{proof}

\begin{acknowledgements}
We thank Burton Newman for pointing out a mistake in an earlier version of this paper, and the anonymous referees for their corrections and for suggesting many improvements both in the presentation and the content of the paper.
\end{acknowledgements}

\section*{Appendix}
Table 1 gives examples for every case possible from Theorem \ref{main}, excluding those when the choice of $E(\Q)_{tors}$ uniquely determines the number of twists with large torsion. The second column gives an example of such an elliptic curve, the third column gives the number of quadratic twists of $E$, and in the fourth column all the $d$-s such that $E^d(\Q)$ has large torsion are listed, with $E^d(\Q)_{tors}$ given in brackets for every listed $d$.

\begin{center}
\begin{tabular}{|c|c|c|c|}
\hline
$E(\Q)_{tors}$ & $E$ & $\#$  & $d (E^d(\Q)_{tors})$ \\
\hline
$C_1$ & 11a2 & $0$& -\\
$C_1$ & 832f1 & $1$ & $2(C_7)$\\
$C_1$ & 1600c1 & $2$ & $2(C_5), 10(C_3)$\\
$C_2$ & 52a1 & 0& - \\
$C_2$ & 4800l3 & 1 & $2(C_{10})$\\
$C_2$ & 75b4 & 2& $5(C_4),3(C_4)$\\
$C_2$ & 450g1 & 3& $-3(C_4),5(C_4),-15(C_6)$\\
$C_3$ & 19a3 & 0& - \\
$C_3$ & 19a1 & 1 & $-3(C_3)$ \\
$C_4$ & 15a7 & 1& $15(C_4)$ \\
$C_4$ & 150c1 & 2 & $-15(C_4),5(C_6)$ \\
$C_5$ & 11a1 & 0& - \\
$C_5$ & 50b1 & 1 & $5(C_3)$  \\
$C_6$ & 20a1 & 0& - \\
$C_6$ & 14a2  & 1 &$-3(C_6)$ \\
$C_6$ & 30a1 & 2& $5(C_4),-3(C_4)$ \\
$C_2\oplus C_2$ &120b2& 0& - \\
$C_2\oplus C_2$ &  150c2& 1 & $5(C_2\oplus C_6)$  \\
$C_2\oplus C_4$ &24a1& 0& - \\
$C_2\oplus C_4$ &  15a1& 1 & $-1(C_2\oplus C_4)$  \\
\hline

\end{tabular}
\vspace{0.2cm}
\\
Table 1.
\end{center}
\newpage
Table 2 gives examples for every case possible from Theorem \ref{main2}, excluding those when the choice of $E(\Q)_{tors}$ uniquely determines the number of quadratic fields in which the torsion of $E$ grows. The values in the columns are listed in a similar manner as in Table 1.

\begin{center}
\begin{tabular}{|c|c|c|c|}
\hline
$E(\Q)_{tors}$ & $E$ & $g(E)$  & $d (E(\Q (\sqrt d))_{tors})$ \\
\hline
$C_1$ & 11a2 & $0$& -\\
$C_1$ & 832f1 & $1$ & $2(C_7)$ \\
$C_1$ & 1600c1 & $2$ & $2(C_5), 10(C_3)$\\
$C_2$ & 52a1 & 1& $-1(C_2\oplus C_2)$ \\
$C_2$ & 4800l3 & 2 & $ -15(C_2\oplus C_2), 2(C_{10})$ \\
$C_2$ & 75b4 & 3& $15(C_2 \oplus C_2), 5(C_8), 3(C_4)$\\
$C_2$ & 2880n1 & 4& $-6(C_6), 2(C_4), -30(C_4), -15(C_2\oplus C_2)$, \\
$C_3$ & 19a3 & 0& - \\
$C_3$ & 19a1 & 1 & $-3(C_3\oplus C_3)$ \\
$C_4$ & 222c4 & 1& $3(C_2\oplus C_4)$ \\
$C_4$ & 90c1 & 2 & $-15(C_2\oplus C_4),-3(C_{12})$ \\
$C_4$ & 15a7 & 3 & $5(C_8),3(C_8),15(C_2\oplus C_4)$\\
$C_5$ & 11a1 & 0& - \\
$C_5$ & 50b1 & 1 & $5(C_{15})$  \\
$C_6$ & 20a1 & 1& $-1(C_2\oplus C_6)$ \\
$C_6$ & 14a2  & 2 &$2(C_2\oplus C_6),-3(C_3\oplus C_6)$ \\
$C_6$ & 30a1 & 3& $5(C_{12}),-3(C_{12}),-15(C_2\oplus C_6)$  \\
$C_8$ & 15a4 & 1 & $-1(C_2\oplus C_8)$\\
$C_8$ & 210e1 & 3 &  $-7(C_2\oplus C_8),-15(C_{16}), 105(C_{16})$\\
$C_2\oplus C_2$ & 120b2& 0& - \\
$C_2\oplus C_2$ &  33a1 & 1 & $-1(C_2\oplus C_4)$  \\
$C_2\oplus C_2$ & 960o2 & 2& $6(C_2\oplus C_4),-2(C_2\oplus C_6)$ \\
$C_2\oplus C_2$ & 15a2 & 3& $5(C_2\oplus C_4),-5(C_2\oplus C_4), -1(C_2\oplus C_4)$ \\
$C_2\oplus C_4$ & 24a1 & 0 &  - \\
$C_2\oplus C_4$ & 21a1 & 1 & $-3(C_2\oplus C_8)$ \\
$C_2\oplus C_4$ & 15a1 & 2 & $-1(C_4 \oplus C_4),5(C_2\oplus C_8)$ \\
$C_2\oplus C_4$ & 210e3 & 3 & $-1(C_4 \oplus C_4), 6(C_2\oplus C_8),-6(C_2\oplus C_8)$ \\
$C_2\oplus C_6$ & 30a2 & 0 &  - \\
$C_2\oplus C_6$ &  90c6 & 1 & $6(C_2\oplus C_{12})$ \\
\hline

\end{tabular}
\vspace{0.2cm}
\\
Table 2.
\end{center}

\end{document}